\documentclass{article}
\usepackage{psfrag}
\usepackage[parfill]{parskip}
\usepackage{float, graphicx}
\usepackage[]{epsfig}
\usepackage{amsmath, amsthm, amssymb,}
\usepackage{epsfig}
\usepackage{verbatim}
\usepackage{multicol}
\usepackage{url}
\usepackage{latexsym}
\usepackage{mathrsfs}
\usepackage[colorlinks, bookmarks=true]{hyperref}
\usepackage{amsmath}
\usepackage{enumerate}
\usepackage[normalem]{ulem}
\usepackage{bm}
\usepackage{dsfont}
\usepackage{stmaryrd}
\usepackage{enumitem}
\usepackage{cite}
\usepackage{color}
\usepackage{xcolor}
\usepackage{pgf,tikz}
\usepackage{caption}
\usepackage{subcaption}

\makeatletter

\def\@settitle{\begin{center}%
    \bfseries
 \normalfont\LARGE\@title
  \end{center}%
}
\def\@setauthors{\begin{center}%
 \normalsize\@author
  \end{center}%
}

\makeatother

\numberwithin{equation}{section}

\renewcommand{\cal}{\mathcal}

\newcommand{\cD}{{\cal D}}
\newcommand{\cE}{{\cal E}}
\newcommand{\cG}{{\cal G}}
\newcommand\cH{{\mathcal H}}

\newcommand{\cL}{{\cal L}}
\newcommand{\cM}{{\cal M}}

\newcommand{\fd}{{\mathfrak d}}
\newcommand{\fj}{{\mathfrak j}}

\newcommand{\fe}{{\mathfrak e}}

\newcommand{\bmh}{{\bm{h}}}

\newcommand{\bmu}{{\bm{u}}}

\newcommand{\rd}{{\rm d}}

\newcommand{\ri}{\mathrm{i}}

\newcommand{\bE}{\mathbb{E}}

\newcommand{\bP}{\mathbb{P}}

\newcommand{\bR}{{\mathbb R}}

\newcommand{\bZ}{\mathbb{Z}}

\newcommand{\al}{\alpha}

\newcommand{\la}{\lambda}

\DeclareMathOperator{\diag}{diag}

\DeclareMathOperator{\OO}{O}
\DeclareMathOperator{\oo}{o}

\renewcommand{\Re}{\mathop{\mathrm{Re}}}
\renewcommand{\Im}{\mathop{\mathrm{Im}}}

\renewcommand{\leq}{\leqslant}
\renewcommand{\geq}{\geqslant}
\newcommand{\floor}[1] {\lfloor {#1} \rfloor}

\newcommand{\qq}[1]{[\![{#1}]\!]}

\newcommand{\beq}{\begin{equation}}
\newcommand{\eeq}{\end{equation}}

\theoremstyle{plain} 
\newtheorem{theorem}{Theorem}[section]
\newtheorem*{theorem*}{Theorem}

\newtheorem*{lemma*}{Lemma}
\newtheorem{corollary}[theorem]{Corollary}
\newtheorem*{corollary*}{Corollary}
\newtheorem{proposition}[theorem]{Proposition}
\newtheorem*{proposition*}{Proposition}

\newtheorem*{assumption*}{Assumption}
\newtheorem{claim}[theorem]{Claim}

\newtheorem{definition}[theorem]{Definition}
\newtheorem*{definition*}{Definition}

\newtheorem*{example*}{Example}
\newtheorem{remark}[theorem]{Remark}

\newtheorem*{remark*}{Remark}
\newtheorem*{remarks*}{Remarks}

\def\author#1{\par
    {\centering{\authorfont#1}\par\vspace*{0.05in}}
}

\def\titlefont{\fontsize{13}{15}\bfseries\boldmath\selectfont\centering{}}
\def\authorfont{\fontsize{13}{15}}

\let\affiliationfont\rhfont

\def\address#1{\par
    {\centering{\affiliationfont#1\par}}\par\vspace*{11pt}
}

\def\body{
\setcounter{footnote}{0}
\def\thefootnote{\alph{footnote}}
\def\@makefnmark{{$^{\rm \@thefnmark}$}}
}

\def\title#1{
    \thispagestyle{plain}
    \vspace*{-14pt}
    \vskip 79pt
    {\centering{\titlefont #1\par}}%
    \vskip 1em
}

\begin{document}

\title{Eigenvalues for the Minors of Wigner Matrices}

\vspace{1.2cm}

 \author{Jiaoyang Huang}
\address{Harvard University\\
   E-mail: jiaoyang@math.harvard.edu}

~\vspace{0.3cm}

\begin{abstract}
The eigenvalues for the minors of real symmetric ($\beta=1$) and complex Hermitian ($\beta=2$) Wigner matrices form the Wigner corner process, which is a multilevel interlacing particle system.
In this paper, we study the microscopic scaling limit of the Wigner corner process both near the spectral edge and in the bulk, and prove they are universal.
We show: (\rm i)
Near the spectral edge, the corner process exhibit a decoupling phenomenon, as first observed in \cite{MR3558206}. Individual extreme particles have Tracy-Widom$_{\beta}$ distribution; the spacings between the extremal particles on adjacent levels converge to independent Gamma distributions in a much smaller scale.
(\rm ii)
In the bulk, the microscopic scaling limit of the Wigner corner process is given by the bead process for general Sine$_\beta$ process, as constructed recently in \cite{NaVi}.
\end{abstract}

\section{Introduction}
Wigner matrices were introduced by E. Wigner to model the nuclei of heavy atoms in \cite{MR0077805,MR0083848}. He postulated that the gaps between the energy levels of large quantum systems should resemble the eigenvalue gaps of a random matrix resemble, which should depend only on the symmetry class of the model and are independent of the details of the matrix ensemble. Since then, the universality phenomenon of the local eigenvalue statistics has been a central subject of study in random matrix theory. The universality of random matrices can be roughly divided into the edge universality near the spectral edge and the bulk universality in the interior of the spectrum.  The edge universality means the joint law of extreme eigenvalues of Wigner matrices converges to the Tracy-Widom$_\beta$ distribution, which was first identified by Tracy and Widom in \cite{MR1257246, MR1385083}. For the bulk universality, Wigner's original point of view asked for the universality of the eigenvalue gap distributions. Later, Mehta formalized the bulk universality conjecture in his book \cite{MR2129906} and stated that the correlation functions of Wigner matrices should coincide with those of Gaussian orthognal/unitary ensemble (GOE/GUE) asymptotically, which are characterized by Sine$_\beta$ process.

Over the past three decades, spectacular progress on edge and bulk universality conjecture for Wigner matrices has been made. The edge universality was proven first via the moment method and its various generalization \cite{MR1647832,MR1727234, MR3729037} and later by the comparison argument \cite{MR2871147, MR2669449}. The bulk universality for Wigner matrices of all symmetry classes was proven in \cite{MR2661171,MR3306005,MR3372074} for the eigenvalue gap universality, and in \cite{MR1810949, MR2662426, MR2810797, MR2919197, MR2661171,MR3541852,MR3914908} for the universality of correlation functions.
Later, edge and bulk universality conjecture was proven for more general Wigner type random matrices in \cite{MR2847916, MR2981427} for generalized Wigner matrices, in \cite{MR3405746, MR3502606} for deformed Wigner matrices, in \cite{MR3429490,MR2964770,MR3729611} for sparse random matrices, in \cite{MR3916110, ALY} for heavy tailed random matrices and in \cite{MR3941370,MR3949269,MR3629874, AEKS} for correlated random matrices.

The GUE corner process, which is a multilevel interlacing particle process, was first studied by Y. Baryshnikov \cite{MR1818248}, where he showed the largest eigenvalues of minors of GUE have the same law as Brownian last passage percolation. Later, it was proven in \cite{MR2268547} that GUE corner process is a determinantal point process, and converges to the bead process as introduced in \cite{MR2489161}. More general random matrix corner processes with determinantal structure were studied in \cite{MR3149438}. The Hermite $\beta$ corner process as the $\beta$ analogue of GUE corner process was introduced by V. Gorin and M. Shkolnikov \cite{MR3418747}, and they showed in \cite{MR3558206} the spacings between the extremal particles on adjacent levels converges to independent Gamma distributions.  The bulk scaling limit of Hermite $\beta$ corner process was understood recently. In \cite{NaVi},  J. Najnudel and B. Vir{\'a}g constructed the bead process for general Sine$_\beta$ process, and proved this process is the microscopic scaling limit in the bulk of the Hermite $\beta$ corner process.  In the current work, we study the eigenvalues for the minors of Wigner matrices, called the Wigner corner process. When restricted to a line, the edge and bulk universality results indicate that the Wigner corner process is characterized by the Tracy-Widom$_\beta$ distribution near the spectral edge, and converges to the Sine$_\beta$ process in the bulk. In the present work we study the joint asymptotic behaviors of the Wigner corner process.  We prove that the scaling limits of Wigner corner process near the spectral edge and in the bulk are universal as long as the matrix entries have finite fourth moment. Near the spectral edge the Wigner corner process exhibit a decoupling phenomenon (see Figure \ref{fig}):  Individual extreme particles have Tracy-Widom$_{\beta}$ distribution; the spacings between the extremal particles on adjacent levels converge to independent Gamma distributions in a much smaller scale.  This answers an open problem \cite[Problem 8]{Open} by Vu.  The Wigner corner process near the spectral edge in a different scaling scale was studied in \cite{MR3403994}. In the bulk, the microscopic scaling limit of the Wigner corner process is given by the bead process for general Sine$_\beta$ process, as constructed in \cite{NaVi}.

\noindent\textbf{Acknowledgement.} I am thankful to Benjamin Landon and Joseph Najnudel for helpful discussions.

\subsection{Main Results}\label{s:main}
In this work, we consider the following class of random matrices, called Wigner matrices.
\begin{definition}\label{def:gW}
A Wigner matrix $H_N=[h_{ij}]_{1\leq i,j\leq N}$ is a real symmetric or complex Hermitian $N\times N$ matrix whose upper-triangular elements $h_{ij}=\bar h_{ji}$, $i\leq j$ are independent random variables with mean zero and variances $\bE[|h_{ij}|^2]=1$.
In the Hermitian case, we assume $\Re[h_{ij}], \Im[h_{ij}]$ are independent and 
$\bE[\Re[h_{ij}]^2]=\bE[\Im[h_{ij}]^2]=1/2$ for $i<j$.
We also assume that the matrix elements $h_{ij}$ have finite fourth moment
\begin{align*}
\bE[|h_{ij}|^4]<\infty.
\end{align*}
\end{definition}

\begin{remark}
We restrict ourselves to Wigner matrices whose entries have bounded fourth moment. This is the minimal assumption that the edge universality of Wigner matrices holds \cite{MR3161313}. However the bulk universality is more robust. We believe our Theorem \ref{t:bulk} can be possibly generalized to sparse random matrices e.g. \cite{MR3429490,MR2964770}, and heavy tailed random matrices e.g. \cite{MR3916110, ALY}.
\end{remark}

Fix a large integer $K\geq 0$. Let $H_{N+K}$ be a Wigner matrix as in Definition \ref{def:gW}. For any integers $0\leq s\leq K$, we denote $H^{(s)}_{N+K}=[h_{ij}]_{1\leq i,j\leq N+s}$ the top-left $(N+s)\times (N+s)$ minor of $H_{N+K}$, with eigenvalues $\lambda_1^{(s)}\leq\la_2^{(s)}\leq \cdots \leq \la_{N+s}^{(s)}$, and corresponding normalized eigenvectors $\bmu_1^{(s)},\bmu_2^{(s)},\cdots, \bmu_{N+s}^{(s)}$. Then $H_{N+K}^{(s)}$ has spectral decomposition $H_{N+K}^{(s)}=U^{(s)} \Lambda^{(s)} (U^{(s)})^*$, where $\Lambda^{(s)}=\diag\{\la_1^{(s)}, \la_2^{(s)}, \cdots, \la_{N+s}^{(s)}\}$, and $U^{(s)}=[\bmu_1^{(s)}, \bmu_2^{(s)},\cdots,\bmu_{N+s}^{(s)}]$. The eigenvalues of $H_{N+K}^{(s)}$ interlace with the eigenvalues of $H_{N+K}^{(s+1)}$: 
\begin{align*}
\la_1^{(s+1)}\leq \la_1^{(s)}\leq \la_2^{(s+1)}\leq \la_2^{(s)}\leq \cdots  \leq \la_{N+s}^{(s+1)}\leq \la_{N+s}^{(s)}\leq\la_{N+s+1}^{(s+1)}.
\end{align*}

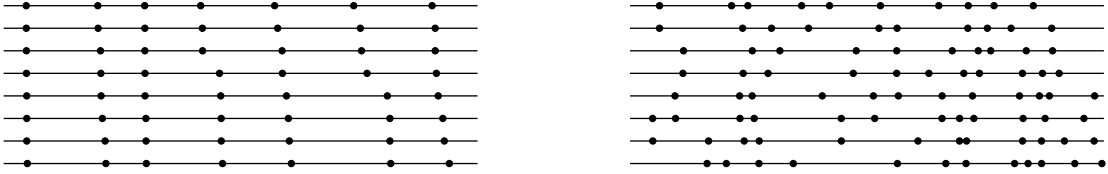
\begin{figure}[h]
\begin{subfigure}{0.5\textwidth}
\centering
\begin{tikzpicture}[scale=0.3]
\foreach \a in {0,1,2,3,4,5,6,7}
\draw [line width=0.5pt] (-1,\a)-- (20.,\a);

\foreach \a in {0.0000,3.1731,5.2498,7.7312,11.0081,14.5133,17.9859}
\draw[fill=black] ( \a,7) circle (4pt);
\foreach \a in {0.0024,3.1863,5.2569,7.8066,11.1411,14.8083,18.1232}
\draw[fill=black] ( \a,6) circle (4pt);
\foreach \a in {0.0082,3.2907,5.2588,7.8152,11.3392,14.8583,18.1242}
\draw[fill=black] ( \a,5) circle (4pt);
\foreach \a in {0.0106,3.3085,5.2631,8.5652,11.3521,15.1045,18.1771}
\draw[fill=black] ( \a,4) circle (4pt);
\foreach \a in {0.0234,3.3176,5.2647,8.6201,11.5327,16.0021,18.2554}
\draw[fill=black] ( \a,3) circle (4pt);
\foreach \a in {0.0237,3.3758,5.3056,8.6406,11.6189,16.1164,18.4613}
\draw[fill=black] ( \a,2) circle (4pt);
\foreach \a in {0.0248,3.4953,5.3095,8.6423,11.6581,16.1173,18.5255}
\draw[fill=black] ( \a,1) circle (4pt);
\foreach \a in {0.0496,3.5313,5.3237,8.6992,11.7483,16.1517,18.7498}
\draw[fill=black] ( \a,0) circle (4pt);
\end{tikzpicture}
\end{subfigure}
\begin{subfigure}{0.5\textwidth}
\centering
\begin{tikzpicture}[scale=0.3]
\foreach \a in {0,1,2,3,4,5,6,7}
\draw [line width=0.5pt] (-1,\a)-- (20.,\a);

This is the 0-th minor0.3071,3.5005,4.2224,6.6063,7.8410,10.0985,12.6823,13.9890,15.1321,16.8714,
This is the 1-th minor0.3078,3.9910,5.2710,6.9133,10.0325,10.8364,13.9702,14.8368,15.8864,17.6866,
This is the 2-th minor1.3720,4.4187,5.6423,9.0243,10.8177,13.2782,14.4312,14.9862,16.5588,17.7235,
This is the 3-th minor1.3453,4.0154,5.1170,8.8932,10.8172,12.2446,13.7906,14.4845,16.3900,17.2770,18.0172,
This is the 4-th minor0.9945,3.8579,4.4127,7.5175,9.7865,10.8922,12.8296,14.1775,16.2590,17.1423,17.5848,19.5894,
This is the 5-th minor0.0024,1.0170,3.8722,4.5043,8.3617,9.8430,12.8260,13.6057,14.2365,16.4184,17.3934,19.1129,
This is the 6-th minor0.0108,2.4814,4.0594,4.7250,8.3665,11.7608,13.6014,13.9148,16.3947,17.2351,18.2512,19.5654,
This is the 7-th minor2.4148,3.2695,4.7138,6.2364,10.8499,12.9958,13.8907,16.0434,16.6399,17.2378,18.7095,19.9111,

\foreach \a in {0.3071,3.5005,4.2224,6.6063,7.8410,10.0985,12.6823,13.9890,15.1321,16.8714}
\draw[fill=black] ( \a,7) circle (4pt);
\foreach \a in {0.3078,3.9910,5.2710,6.9133,10.0325,10.8364,13.9702,14.8368,15.8864,17.6866}
\draw[fill=black] ( \a,6) circle (4pt);
\foreach \a in {1.3720,4.4187,5.6423,9.0243,10.8177,13.2782,14.4312,14.9862,16.5588,17.7235}
\draw[fill=black] ( \a,5) circle (4pt);
\foreach \a in {1.3453,4.0154,5.1170,8.8932,10.8172,12.2446,13.7906,14.4845,16.3900,17.2770,18.0172}
\draw[fill=black] ( \a,4) circle (4pt);
\foreach \a in {0.9945,3.8579,4.4127,7.5175,9.7865,10.8922,12.8296,14.1775,16.2590,17.1423,17.5848,19.5894}
\draw[fill=black] ( \a,3) circle (4pt);
\foreach \a in {0.0024,1.0170,3.8722,4.5043,8.3617,9.8430,12.8260,13.6057,14.2365,16.4184,17.3934,19.1129}
\draw[fill=black] ( \a,2) circle (4pt);
\foreach \a in {0.0108,2.4814,4.0594,4.7250,8.3665,11.7608,13.6014,13.9148,16.3947,17.2351,18.2512,19.5654}
\draw[fill=black] ( \a,1) circle (4pt);
\foreach \a in {2.4148,3.2695,4.7138,6.2364,10.8499,12.9958,13.8907,16.0434,16.6399,17.2378,18.7095,19.9111}
\draw[fill=black] ( \a,0) circle (4pt);
\end{tikzpicture}
\end{subfigure}

\label{fig}
\caption{The left pane is the eigenvalues for the minors of a $5000\times 5000$ Wigner matrix near the spectral edge.  Individual eigenvalues have Tracy-Widom$_{\beta}$ distribution on the scale $N^{-1/6}$; the spacings between the eigenvalues on adjacent levels converge to independent Gamma distributions on a much smaller scale , i.e. $N^{-1/2}$. The right pane is the eigenvalues for the minors of a $5000\times 5000$ Wigner matrix in the bulk.
The limit of the process converges to the bead process for GOE/GUE process on the scale $N^{-1/2}$  }
\label{fig}
\end{figure}

The first main result of our paper studies the asymptotic behaviors of the eigenvalues of the minors $H^{(s)}_{N+K}$ for $0\leq s\leq K$, around the spectral edge.  This answers an open problem \cite[Problem 8]{Open} by Vu.
 
\begin{theorem}[Edge Case]\label{t:edge}
Fix integers $K,\ell\geq 1$. Let $H_{N+K}$ be a Wigner matrix (as in Definition \ref{def:gW}). The vector 
\begin{align}\label{e:TWbeta}
N^{1/6}(\la_1^{(0)}+2\sqrt{N}),\; N^{1/6}(\la_2^{(0)}+2\sqrt{N}),\;\cdots\;, N^{1/6}(\la_\ell^{(0)}+2\sqrt{N}),
\end{align}
converges in distribution in the limit $N\rightarrow \infty$ to the Tracy-Widom$_\beta$ distribution.
The array 
\begin{align}\label{e:Gammabeta}
\left(\sqrt{N}(\la_1^{(s)}-\la_1^{(s+1)}),\; \sqrt{N}(\la_2^{(s)}-\la_2^{(s+1)}),\;\cdots,\; \sqrt{N}(\la_\ell^{(s)}-\la_\ell^{(s+1)})\right)_{0\leq s\leq K-1},
\end{align}
converges in distribution in the limit $N\rightarrow \infty$ to a random array with independent Gamma distributed entries, which has density
$
C_\beta x^{\beta/2-1}e^{-\beta x/2 }.
$
Moreover, \eqref{e:TWbeta} and \eqref{e:Gammabeta} are asymptotically independent.
The same statement holds for the right spectral edge.
\end{theorem}

Let $\cL$ be the set of discrete point configurations $\{\mu_j\}_{j\in \bZ}$ of $\bR$ (where the labeling is uniquely determined by setting $\mu_{-1}<0\leq \mu_0$) unbounded from above and from below, such that for $x$ going to infinity, and for all $a,b\in \bR$, $|\{j: \mu_j\in [0, x+a]\}|-|\{j:\mu_j\in [-x+b, 0]\}|=\OO(x/\log^2 x)$. The $\sigma$-algebra on $\cL$ is generated by the maps $L\mapsto |L\cap I|$ for all open, bounded intervals $I\subset \bR$. Let $\cal G$ be the set of doubly infinite sequence $\{\gamma_j\}_{j\in \bZ}$ such that for $k$ going to infinity
\begin{align}\label{e:sumcon}
\sum_{0\leq j\leq k}\gamma_j=k+\OO(k/\log^2 k),\quad
\sum_{-k\leq j\leq 0}\gamma_j=k+\OO(k/\log^2 k).
\end{align}
The $\sigma$-algebra on $\cal G$  is generated by the coordinate maps $\gamma_j$ for $j\in \bZ$.

For $\{\mu_j\}_{j\in \bZ}\in \cL$ and $\{\gamma_j\}_{j\in \bZ}\in \cal G$, it is proven in \cite[Theorem 3]{NaVi} the following Stieltjes transform of the weighted measure $\sum_{j\in \bZ}\gamma_j \delta_{\mu_j}$ is well-defined,
\begin{align*}
S_{\sum_{j\in \bZ}\gamma_j \delta_{\mu_j}}(z)
=\sum_{j\in \bZ}\frac{\gamma_j}{\mu_j-z}
=\lim_{x\rightarrow \infty}\sum_{j\in \bZ :\mu_j\in [-x,x]}\frac{\gamma_j}{\mu_j-z}.
\end{align*}
Let $\cD$ be the map from $\cL\times \cG\times \bR$ to $\cL$ , defined by 
\begin{align*}
\cD(\{\mu_j\}_{j\in \bZ}, \{\gamma_j\}_{j\in \bZ}, h)=S^{-1}_{\sum_{j\in \bR}\gamma_j\delta_{\mu_j}}(h).
\end{align*}
The operator $\cD$ is measurable and it is indeed a map to $\cL$. Let $\Pi$ be the probability measure on $\cG\times \bR$ under which the random variables $h$ and $\{\gamma_j\}_{j\in \bZ}$ are independent, $\gamma_j$ follows Gamma distribution with density $C_\beta x^{\beta/2-1}e^{-\beta x/2}$. The operator $\cD$ defines a Markov chain $\{\Theta^{(s)}\}_{s\geq 0}$ on $\cL$,
\begin{align}\label{e:markov}
\Theta^{(s+1)}=\cD(\Theta^{(s)}, \Gamma^{(s)},h), \quad s=0,1,2,\cdots.
\end{align}
where $(\Gamma^{(s)},h)$ is sampled from $\Pi$. It is proven in \cite[Theorem 9]{NaVi}, the Sine$_\beta$ point process is supported on $\cL$, and is an invariant measure for the Markov chain \eqref{e:markov}. The second main result of our paper studies the asymptotic behaviors of the eigenvalues of the minors $H^{(s)}_{N+K}$ for $0\leq s\leq K$ in the bulk of the spectrum.   
\begin{theorem}[Bulk Case]\label{t:bulk}
Fix integer $K\geq 1$, and energy $E\in (-2,2)$. Let $H_{N+K}$ be a Wigner matrix (as in Definition \ref{def:gW}). The point process 
\begin{align*}
\Theta^{(s)}_{N}=\sum_{j=1}^{N+s}\delta_{\sqrt{N(4-E^2)}(\la_j^{(s)}-E\sqrt N)}, \quad 0\leq s\leq K,
\end{align*}
forms a Markov chain, and converges in distribution in the limit $N\rightarrow \infty$ to the Markov chain \eqref{e:markov} for the level
\begin{align*}
h=-\frac{E}{2\sqrt{4-E^2}}.
\end{align*}
\end{theorem}

%
%
%
%
%

\section{Preliminaries}

In the paper, we write $a_N=\OO(b_N)$ or $a_N\lesssim b_N$ if there exists some universal constant $C$ such that $|a_N|\leq C b_N$. We write $a_N=\oo(b_N)$, or $a_N\ll b_N$ if the ratio $|a_N|/b_N\rightarrow 0$ as $N$ goes to infinity. We denote the set $\{1, 2,\cdots, N\}$ by $\qq{1,N}$. We say an event $\Omega_N$ holds with high probability, if $\bP(\Omega_N)=1-\oo(1)$. We write $a_N=\oo_\bP(b_N)$ if with high probability $a_N=\oo(b_N)$ holds. We write $a_N=\OO_\bP(b_N)$ if with high probability $a_N=\OO(b_N)$ holds.
We denote the semi-circle distribution as 
\begin{align*}
\rho_{sc}(x)=\frac{\sqrt{4-x^2}}{2\pi}, \quad x\in [-2,2], \quad \mu_{sc}[a,b]=\int_a^b \rho_{sc}(x)\rd x.
\end{align*}
We denote $\cM(\bR)$ the space of locally finite measures on the Borel sets of $\bR$, endowed with the topology of \emph{locally} weak convergence, i.e. a sequence of locally finite measures $\{\mu_N\}$ converge to $\mu$ if $\lim_{N\rightarrow\infty}\int f \rd \mu_N = \int f \rd \mu$ for all continuous compactly supported functions $f: \bR\mapsto \bR$.  We recall that  $\cG$ is the set of doubly infinite sequence $\{\gamma_j\}_{j\in \bZ}$ satisfying \eqref{e:sumcon}.
A sequence of doubly infinite sequence $\{\gamma_j^N\}_{j\in \bZ}\in \cG$ converges to $\{\gamma_j\}_{j\in \bZ}$ in distribution, if  $\{\gamma_j^N\}_{j\in \bZ}$ converges to $\{\gamma_j\}_{j\in \bZ}$ coordinatewise in distribution. We remark that with the topology of locally weak convergence, both $\cM(\bR)$ and $\cG$ form polish spaces.

\subsection{Rigidity of eigenvalues and delocalization of eigenvectors}
For Wigner matrices with high probability, the eigenvalues are rigid, i.e. they are close to the classical eigenvalue locations of semi-circle distribution, and eigenvectors are completely delocalized. The rigidity of eigenvalues and delocalization of eigenvectors were first proven in \cite{MR2871147, MR2481753}
under the assumption that the matrix entries have subexponential decay. Later the same statement was proven in \cite[Theorem 1.3]{GNT} under the assumption that the matrix entries have bounded fourth moment. 
\begin{theorem}[{\cite[Theorem 1.3]{GNT}}]\label{t:rigid}
Let $H_{N}$ be a Wigner matrix (as in Definition \ref{def:gW}), with eigenvalues $\la_1\leq \la_2\leq \cdots\leq \la_N$ and normalized eigenvectors $\bmu_1, \bmu_2,\cdots, \bmu_N$. For any $D\geq 0$, with probability at least $1-N^{-D}$ the following holds:
\begin{enumerate}
\item The eigenvalues of $H_{N}$ are rigid with respect to the semi-circle distribution
\begin{align}\label{e:rigid}
\left||\{j: \la_j\in (-\infty, E\sqrt{N}]\}|-N\int_{-\infty}^{E}\rho_{sc}(x)\rd x\right|\leq (\log N)^{\OO(1)}, \quad E\in \bR,
\end{align}
\item The eigenvectors of $H_{N}$ are completely delocalized
\begin{align}\label{e:delocalize}
\|\bmu_i\|_{\infty}\leq \frac{(\log N)^{\OO(1)}}{\sqrt N},  \quad 1\leq i\leq N,
\end{align}
\end{enumerate}
where the implicit constants depend on $D$.
\end{theorem}

\subsection{Master equation}
Given an $(n+1)\times (n+1)$ matrix $A_{n+1}$, we denote its top-left $n\times n$ minor by $A_n$. The eigenvalues of $A_{n+1}$ satisfy a master equation in terms of the eigenvalues and eigenvectors of $A_n$.  We denote the eigenvalues of $A_n$ as $\la_1\leq \la_2\leq \cdots\leq \la_n$, with corresponding normalized eigenvectors $\bmu_1,\bmu_2,\cdots, \bmu_n$. Then $A_n$ has spectral decomposition $A_n=U \Lambda U^*$, where $\Lambda=\diag\{\la_1, \la_2, \cdots, \la_n\}$, and $U=[\bmu_1, \bmu_2,\cdots,\bmu_n]$. We can write $A_n$ as a $2\times 2$ block matrix
\begin{align*}
A_{n+1}=\left[
\begin{array}{cc}
U\Lambda U^*, &\bm g\\
\bm g^*, &g_{n+1}
\end{array}
\right], \quad \bm g=(g_1, g_2, \cdots, g_n)^t.
\end{align*}
The determinant formula for block matrices gives
\begin{align*}
\det (A_{n+1}-z)
=\det(U\Lambda U^*-z)\det\left(g_{n+1}-z-\bm g^* (U\Lambda U^*-z)^{-1}\bm g\right ).
\end{align*}
Eigenvalues of $A_{n+1}$ are zeros of the meromorphic function
\begin{align}\label{e:meq}
g_{n+1}-z-\bm g^* (U\Lambda U^*-z)^{-1}\bm g
=g_{n+1}-z-\sum_{i=1}^n \frac{|(U^*\bm g)_i|^2}{\la_i-z}
=g_{n+1}-z-\sum_{i=1}^n \frac{|\langle\bmu_i,\bm g\rangle|^2}{\la_i-z}.
\end{align}
The expression \eqref{e:meq} will be used repeatedly in the rest of this paper to get the estimates of eigenvalues of $A_{n+1}$ conditioning on the submatrix $A_n$.

\section{Edge Case}

The first part  \eqref{e:TWbeta} of Theorem \ref{t:edge} is the edge universality statement. Under the assumption that matrix entries have bounded fourth moment, edge universality was proven in \cite[Theorem 1.2]{MR3161313}. In fact, edge universality requires slightly weaker assumption, i.e. $\lim_{s\rightarrow \infty}s^4 \bP(|h_{ij}|>s)=0$.

\begin{theorem}[Edge Universality,{\cite[Theorem 1.2]{MR3161313}}]\label{t:edgeU}
Let $H_{N}$ be a Wigner matrix (as in Definition \ref{def:gW}) with eigenvalues $\la_1\leq \la_2\leq\cdots\leq \la_N$. Its rescaled extreme eigenvalues  
\begin{align*}
N^{1/6}(\la_1+2\sqrt{N}),\; N^{1/6}(\la_2+2\sqrt{N}),\;\cdots\;, N^{1/6}(\la_\ell+2\sqrt{N}),
\end{align*}
converge in distribution in the limit $N\rightarrow \infty$ to the Tracy-Widom$_\beta$ distribution. The same statement holds for the right spectral edge.
\end{theorem}

We denote the set $\cH_N$ of real symmetric or complex Hermitian $N\times N$ matrices, such that the following holds
\begin{enumerate}
\item The eigenvalues of $H_{N}$ are rigid with respect to the semi-circle distribution, i.e. \eqref{e:rigid} holds.
\item The eigenvectors of $H_{N}$ are completely delocalized, i.e. \eqref{e:delocalize} holds.
\item The extreme eigenvalues satisfy: $\la_1=-2\sqrt N +\OO((\log N)^{\OO(1)}N^{-1/6})$, and $\la_{i+1}-\la_i\gtrsim (\log N)^{\OO(1)}N^{-1/6}$, for $1\leq i\leq \ell-1$.
\end{enumerate}
It follows from Theorems \ref{t:rigid} and \ref{t:edgeU}, the event $\cH_N$ holds with high probability, i.e. $\bP(H_N\in \cH_N)=1-\oo(1)$. The second part \eqref{e:Gammabeta} of Theorem \ref{t:edge} follows from the following Proposition.

\begin{proposition}\label{p:eigenloc}
Fix integers $K,\ell\geq 1$. Let $H_{N+K}$ be a Wigner matrix (as in Definition \ref{def:gW}).  For any $0\leq s\leq K-1$, conditioning on $H^{(s)}_{N+K}\in \cH_{N+s}$, 
with respect to the randomness of $\bmh^{(s+1)}=(h_{1 N+s+1}, h_{2 N+s+1}, \cdots, h_{N+s N+s+1})^\top$ and $h_{N+s+1N+s+1}$,
\begin{align}\label{e:approx}
\sqrt{N}(\la_i^{(s)}-\la_i^{(s+1)})= \left|\langle \bmu_i^{(s)}, \bmh^{(s+1)}\rangle\right|^2+\oo_\bP(1), \quad 1\leq i\leq \ell.
\end{align}
And the random vector 
\begin{align}\label{e:Gamma}
\left(\left|\langle \bmu_1^{(s)}, \bmh^{(s+1)} \rangle\right|^2, \left|\langle \bmu_2^{(s)}, \bmh^{(s+1)} \rangle\right|^2,\cdots, \left|\langle \bmu_\ell^{(s)}, \bmh^{(s+1)} \rangle\right|^2\right)
\end{align}
converges in distribution in the limit $N\rightarrow \infty$ to a random array with independent Gamma distributed entries, which has density
$
C_\beta x^{\beta/2-1}e^{-\beta x/2 }.
$
\end{proposition}

\begin{proof}
We first prove \eqref{e:Gamma}. For the real symmetric case,  we have
\begin{align*}\begin{split}
&\phantom{{}={}}\bE\left[\exp\left\{\sum_{j=1}^{\ell}\ri t_j\langle\bmu_j^{(s)}, \bmh^{(s+1)}\rangle\right\}\right]
=\bE\left[\exp\left\{\ri\left\langle\sum_{j=1}^{\ell} t_j\bmu_j^{(s)}, \bmh^{(s+1)}\right\rangle\right\}\right]\\
&=\prod_{\alpha=1}^{N+s}\bE\left[1+\ri\sum_{j=1}^{\ell} t_j\bmu_{j\alpha}^{(s)}\bmh_\alpha^{(s+1)}-\frac{1}{2}\left(\sum_{j=1}^{\ell} t_j\bmu_{j\alpha}^{(s)} \bmh_\alpha^{(s+1)}\right)^2+\OO\left(\left|\sum_{j=1}^{\ell} t_j\bmu_{j\alpha}^{(s)} \bmh_\alpha^{(s+1)}\right|^3\right)\right]\\
&=\prod_{\alpha=1}^{N+s}\bE\left(1-\frac{1}{2}\sum_{j=1}^{\ell} \left(t_j\bmu_{j\alpha}^{(s)}\right)^2+\OO\left(\max_{1\leq j\leq \ell} |t_j|^3\frac{(\log N)^{\OO(1)}\ell^3}{N^{3/2}}\right)\right)\\
&=\exp\left\{-\frac{1}{2}\sum_{j=1}^\ell t_j^2+\OO\left(\max_{1\leq j\leq \ell} |t_j|^3\frac{(\log N)^{\OO(1)}\ell^3}{N^{1/2}}\right)\right\},
\end{split}\end{align*}
where in the third line we used that $H_{N+K}^{(s)}\in \cH_{N+s}$, and its eigenvectors are delocalized. Therefore, $\langle\bmu_1^{(s)}, \bmh^{(s+1)} \rangle, \langle \bmu_2^{(s)}, \bmh^{(s+1)} \rangle,\cdots, \langle \bmu_\ell^{(s)}, \bmh^{(s+1)} \rangle$ converges to independent real gaussian random variables, as long as $\ell\ll N^{1/6}$. The claim \eqref{e:Gamma} follows by noticing that the square of a standard gaussian random variable has Gamma distribution, with density $x^{-1/2}e^{-x/2}/2\sqrt{2\pi}$.

Similarly, for the complex Hermitian case, the vector $\langle\bmu_1^{(s)}, \bmh^{(s+1)} \rangle, \langle \bmu_2^{(s)}, \bmh^{(s+1)} \rangle,\cdots, \langle \bmu_\ell^{(s)}, \bmh^{(s+1)} \rangle$ converges to independent complex gaussian random variables, with independent real and imaginary part, each has variance $1/2$. The claim \eqref{e:Gamma} for the complex Hermitian case follows by noticing that the norm square of a complex gaussian random variable has Gamma distribution, with density $e^{-x}$.

In the following we prove \eqref{e:approx}.
By taking $A_{n+1}=H_{N+K}^{(s+1)}$ in the master equation \eqref{e:meq}, then $A_n=H_{N+K}^{(s)}$ and the eigenvalues of $H_{N+K}^{(s+1)}$ are zeros of the meromorphic function 
\begin{align}\label{e:eqroot}
f(z)=h_{N+s+1 N+s+1}-z-\sum_{j=1}^{N+s} \frac{|\langle \bmu_j^{(s)},\bm h^{(s+1)}\rangle|^2}{\la_j^{(s)}-z}.
\end{align} 
More explicitly, $\lambda_i^{(s+1)}$ is the zero of \eqref{e:eqroot} between $\lambda_{i-1}^{(s)}$ and $\lambda_{i}^{(s)}$ (we make the convention that $\lambda_0^{(s)}=-\infty$ and $\la_{N+s+1}^{(s)}=+\infty$).  On the interval $(\la_{i-1}^{(s)}, \la_i^{(s)})$, $f(z)$ is monotone decreasing. To locate the eigenvalue $\la_{i}^{(s+1)}$, we rewrite $f(z)$ as
\begin{align*}
f(z)=\sqrt{N}-\left( \frac{|\langle \bmu_{i-1}^{(s)},\bm h^{(s+1)}\rangle|^2}{\la_{i-1}^{(s)}-z}+ \frac{|\langle \bmu_i^{(s)},\bm h^{(s+1)}\rangle|^2}{\la_i^{(s)}-z}\right)+\cE(z),
\end{align*}
where 
\begin{align}\label{e:error}
\cE(z)=h_{N+s+1N+s+1}-(z+2\sqrt{N})-\left(\sum_{1\leq j\leq N+s\atop j\neq i-1,i}\frac{|\langle \bmu_j^{(s)},\bm h^{(s+1)}\rangle|^2}{\la_j^{(s)}-z}-\sqrt{N}\right)
\end{align}
Using the rigidity of the eigenvalues and delocalization of the eigenvectors as input, we show that with high probability the error term $\cE(z)$ is negligible. We postpone its proof to the end of this section.
\begin{claim}\label{c:error}
Under the assumptions of Proposition \ref{p:eigenloc}, for any $z\in (\la_{i-1}^{(s)}, \la_i^{(s)})$ we have,
\begin{align}\label{e:error}
\bP\left(|\cE(z)|\geq (\log N)^{\OO(1)}N^{1/6}\right)\lesssim (\log N)^{-\OO(1)}.
\end{align}
\end{claim}
As a consequence of Claim \ref{c:error}, any fixed $z\in (\la_{i-1}^{(s)}, \la_{i}^{(s)})$, with high probability
\begin{align*}
f(z)=\tilde f(z)+\OO\left((\log N)^{\OO(1)}N^{1/6}\right),\quad
\tilde f(z)=\sqrt{N}-\left( \frac{|\langle \bmu_{i-1}^{(s)},\bm h^{(s+1)}\rangle|^2}{\la_{i-1}^{(s)}-z}+ \frac{|\langle \bmu_i^{(s)},\bm h^{(s+1)}\rangle|^2}{\la_i^{(s)}-z}\right)
\end{align*}
Thanks to the assumption that $H_{N+K}^{(s)}$ and \eqref{e:Gamma}, with probability $1-\oo(1)$, it holds
\begin{align}\label{e:gapb}
\la_i^{(s)}-\la_{i-1}^{(s)}\gtrsim (\log N)^{\OO(1)}N^{-1/6},
\end{align}
and
\begin{align}\label{e:uhbound}
(\log N)^{-\OO(1)}\leq |\langle \bmu_{i-1}^{(s)},\bm h^{(s+1)}\rangle|^2, |\langle \bmu_i^{(s)},\bm h^{(s+1)}\rangle|^2\leq (\log N)^{\OO(1)}.
\end{align}
On the event that \eqref{e:gapb}, \eqref{e:uhbound} hold, $\tilde f(z)$ is monotone decreasing from $+\infty$ (or $\sqrt{N}$ if $i=1$) to $-\infty$, on the interval $z\in (\la_{i-1}^{(s)}, \la_{i}^{(s)})$. Thus, on the event that \eqref{e:gapb}, \eqref{e:uhbound} hold,  there exists $z_1,z_2=\la_i^{(s)}-|\langle \bmu_i^{(s)}, \bmh^{(s+1)}\rangle|^2/\sqrt{N}+\OO((\log N)^{\OO(1)}N^{-5/6})$, such that 
\begin{align*}
\tilde f(z_1)\gg (\log N)^{\OO(1)}N^{1/6}, \quad \tilde f(z_2)\ll -(\log N)^{\OO(1)}N^{1/6}.
\end{align*}
Combining the above estimates with Claim \ref{c:error}, with high probability we have
\begin{align*}
f(z_1)=\tilde f(z_1)+\cE(z_1)>0, \quad f(z_2)=\tilde f(z_2)+\cE(z_2)<0.
\end{align*} 
Thanks to the monotonicity of $f(z)$, we conclude that with high probability $z_1< \la_i^{(s+1)}<z_2$, and\begin{align*}
\la_i^{(s+1)}=\la_i^{(s)}-\frac{|\langle \bmu_i^{(s)}, \bmh^{(s+1)}\rangle|^2}{\sqrt{N}}+\OO((\log N)^{\OO(1)}N^{-5/6}).
\end{align*}
The claim \eqref{e:approx} follows.
\end{proof}

\begin{proof}[Proof of Claim \ref{c:error}]
Under the assumptions of Proposition \ref{p:eigenloc} that $H_{N+K}^{(s)}\in \cH_{N+s}$, for any $z\in (\la_{i-1}^{(s)}, \la_i^{(s)})$ we have,
\begin{align*}
\bP(|h_{N+s+1N+s+1}|\geq (\log N)^{\OO(1)})\lesssim (\log N)^{-\OO(1)}, \quad |z+2\sqrt{N}|\lesssim (\log N)^{\OO(1)} N^{-1/6}.
\end{align*}
The claim \eqref{e:error} follows from combining the following two estimates.
\begin{align}
\label{e:locallaw}&\sum_{1\leq j\leq N+s\atop j\neq i-1,i}\frac{1}{\la_j^{(s)}-z}
=\sqrt{N}+\OO\left((\log N)^{\OO(1)}N^{1/6}\right),\\
\label{e:concentrate}&\bP\left(\left|\sum_{1\leq j\leq N+s\atop j\neq i-1,i}\frac{|\langle \bmu_j^{(s)},\bm h^{(s+1)}\rangle|^2}{\la_j^{(s)}-z}-\frac{1}{\la_j^{(s)}-z}\right|\geq (\log N)^{\OO(1)} N^{1/6}\right)\lesssim (\log N)^{-\OO(1)}.
\end{align}
Since $H_{N+K}^{(s)}\in \cH_{N+s}$, we have 
$z-\la_{i-2}^{(s)}\gtrsim (\log N)^{-\OO(1)}N^{-1/6}$ and $\la_{i+1}^{(s)}-z\gtrsim (\log N)^{-\OO(1)}N^{-1/6}$. The sum of eigenvalues smaller than $z$ in \eqref{e:locallaw} satisfies 
\begin{align}\label{e:left}
\left|\sum_{1\leq j< i-1}\frac{1}{\la_j^{(s)}-z}\right|
\lesssim \ell (\log N)^{\OO(1)}N^{1/6}.
\end{align}
For the sum of the eigenvalues bigger than $z$, we need to use the rigidity estimates, i.e. \eqref{e:rigid},
\begin{align}\begin{split}\label{e:dif}
&\phantom{{}={}}\left|\sum_{j> i}\frac{1}{\la_j^{(s)}-z}
- N\int_{\la_{i+1}^{(s)}}^\infty \frac{\rho_{sc}(x/\sqrt N)}{x-z}\rd x\right|\\
&=\left|\int_{\la_{i+1}^{(s)}}^\infty\frac{|\{j: \la^{(s)}_j \in [\la_{i+1}^{(s)}, x]\}|-\sqrt N\mu_{sc}[\la_{i+1}^{(s)}/\sqrt N , x/\sqrt N]}{(x-z)^2}\rd x\right|\\
&\leq \int_{\la_{i+1}^{(s)}}^\infty\frac{(\log N)^{\OO(1)}}{|x-z|^2}\rd x
\lesssim (\log N)^{\OO(1)}N^{1/6}.
\end{split}\end{align}
Finally for the integral of the semi-circle distribution, under our assumption $z=-2\sqrt N+\OO((\log N)^{\OO(1)}N^{-1/6})$,  and $\la_{i+1}^{(s)}-z\asymp (\log N)^{\OO(1)}N^{-1/6}$,
\begin{align}\label{e:intsc}
\sqrt N\int_{\la_{i+1}^{(s)}}^\infty \frac{\rho_{sc}(x/\sqrt N)}{x-z}\rd x
=2\sqrt{N}+\OO\left((\log N)^{\OO(1)}N^{1/6}\right).
\end{align}
The first estimate \eqref{e:locallaw} follows from combining \eqref{e:left}, \eqref{e:dif} and \eqref{e:intsc}. 

The second estimate \eqref{e:concentrate} follows from a second moment computation and the Markov inequality.  
\begin{align*}\begin{split}
&\phantom{{}={}}\bE\left[\left|\sum_{1\leq j\leq N+s\atop j\neq i-1,i}\frac{|\langle \bmu_j^{(s)},\bm h^{(s+1)}\rangle|^2}{\la_j^{(s)}-z}-\frac{1}{\la_j^{(s)}-z}\right|^2\right]\\
&=\sum_{1\leq j,j'\leq N+s\atop j,j'\neq i-1,i}\frac{\bE[(|\langle \bmu_j^{(s)},\bm h^{(s+1)}\rangle|^2-1)(|\langle \bmu_{j'}^{(s)},\bm h^{(s+1)}\rangle|^2-1)]}{(\la_j^{(s)}-z)(\la_{j'}^{(s)}- z)}\\
&=\sum_{1\leq j,j'\leq N+s\atop j,j'\neq i-1,i}\frac{(2/\beta)\delta_{jj'}+\OO(\sum_\alpha |\bmu_{j\alpha}^{(s)}|^2 |\bmu_{j'\alpha}^{(s)}|^2)}{(\la_j^{(s)}-z)(\la_{j'}^{(s)}- z)}\\
&=\frac{2}{\beta}\sum_{1\leq j\leq N+s\atop j\neq i-1,i}\frac{1}{(\la_j^{(s)}-z)^2}+\OO\left(\frac{(\log N)^{\OO(1)}}{N}\right)\left(\sum_{1\leq j\leq N+s\atop j\neq i-1,i}\frac{1}{|\la_j^{(s)}-z|}\right)^2\lesssim (\log N)^{\OO(1)}N^{1/3},
\end{split}\end{align*}
where for the third equality, we used the delocalization of eigenvectors, \eqref{e:delocalize}; for the last inequality we used \eqref{e:locallaw}, and $\sum_{1\leq j\leq N+s\atop j\neq i-1,i}(\la_j^{(s)}-z)^{-2}\lesssim (\log N)^{\OO(1)}N^{1/3}$, which can be obtained by a similar argument as for \eqref{e:locallaw}. This finishes the proof of Proposition \ref{e:concentrate} and Claim \ref{c:error}.
\end{proof}

\section{Bulk Case}

When restricted to a line, the law of Wigner corner process is asymptotically the same as that of GOE/GUE, which is characterized by the Sine$_\beta$ process in the bulk. The following bulk universality results follow from \cite{MR3914908} using the rigidity estimates Theorem \ref{t:bulk} as input.
\begin{theorem}[Bulk Universality]\label{t:bulkU}
Let $H_{N}$ be a Wigner matrix (as in Definition \ref{def:gW}), with eigenvalues $\la_1\leq \la_2\leq \cdots \leq \la_N$. For any compactly supported smooth test function $F:\bR^k\mapsto \bR$, and energy level $E$
\begin{align*}\begin{split}
&\phantom{{}={}}\sum_{i_1,i_2,\cdots,i_k\in \qq{1,N}}\bE\left[O\left(\sqrt{N}\la_{i_1}-EN,\sqrt{N}\la_{i_2}-EN, \cdots, \sqrt{N}\la_{i_k}-EN\right)\right]\\
&=\sum_{i_1,i_2,\cdots,i_k\in \qq{1,N}}\bE_{GOE/GUE}\left[O\left(\sqrt{N}\la_{i_1}-EN, \sqrt{N}\la_{i_2}-EN, \cdots, \sqrt{N}\la_{i_k}-EN\right)\right]+\oo(1)
\end{split}\end{align*}

\end{theorem}

In \cite[Theorem 7]{NaVi2}, J. Najnudel and B. Vir{\'a}g derived an optimal bounds on the variance of the number of particles of the Gaussian $\beta$ ensemble (G$\beta$E) in intervals of the real line.
\begin{theorem}[{\cite[Theorem 7]{NaVi2}}]
The Gaussian $\beta$ ensemble with $N$ particles is a probability measure
\begin{align}\label{t:GbEvar}
\bP_{G\beta E}(\la_1,\la_2,\cdots,\la_N)=\frac{1}{Z_{N,\beta}}\prod_{1\leq i<j\leq N}|\la_i-\la_j|^\beta \prod_{i=1}^Ne^{-(\beta/4)\la_i^2}.
\end{align} For any energy levels $E_1<E_2$, we have
\begin{align}\label{e:GbEvar}
\bE_{G\beta E}\left[\left(|\{j:\la_j\in [E_1\sqrt{N}, E_2\sqrt{N}]\}|-N\int_{E_1}^{E_2}\rho_{sc}(x)\rd x\right)^2\right]\lesssim \log(2+N(E_2-E_1)\wedge N).
\end{align}
\end{theorem}
The eigenvalues of GOE/GUE have the same law as Gaussian $\beta$ ensemble corresponding to $\beta=1/\beta=2$ respectively. As a corollary of Theorem \ref{t:bulkU} and Theorem \ref{t:GbEvar}, we have the following bound on the variance of the number of eigenvalues of Wigner matrices in intervals of the real line.
\begin{corollary}\label{c:var}
Let $H_{N}$ be a Wigner matrix (as in Definition \ref{def:gW}), with eigenvalues $\la_1\leq \la_2\leq \cdots \leq \la_N$.  For any energy levels $E_1< E_2$, we have
\begin{align}\label{e:var}
\bE\left[\left(|\{j:\la_j\in [E_1\sqrt{N}, E_2\sqrt{N}]\}|-N\int_{E_1}^{E_2}\rho_{sc}(x)\rd x\right)^2\right]\lesssim (\log(2+N(E_2-E_1)\wedge N))^{\OO(1)}.
\end{align}
\end{corollary}
\begin{proof}
Although theorem \ref{t:bulkU} is stated for test functions which are compactly supported, the statement holds for tests functions with support depending on $N$. The proof in \cite[Theorem 2.2]{MR3914908} carries through when the test functions are supported on $[-N^{\fe}, N^{\fe}]^{k}$ with some sufficiently small $\fe>0$. We first prove the statement \eqref{e:var} for $E_2-E_1\leq N^{\fe-1}$. In this case, the left-hand side of \eqref{e:var} can be replaced by 
the corresponding expectation with respect to $GOE/GUE$, then it 
is a special case of \eqref{e:GbEvar}
\begin{align}\label{e:GbEvar2}
\bE_{GOE/GUE}\left[\left(|\{j:\la_j\in [E_1\sqrt{N}, E_2\sqrt{N}]\}|-N\int_{E_1}^{E_2}\rho_{sc}(x)\rd x\right)^2\right]\lesssim \log(2+N(E_2-E_1)\wedge N).
\end{align}
To carry out the proof, we take a bump function $\phi(x)>0$ supported on $[-1,1]$ with $\int \phi(x)\rd x=1$. Fix any energy levels $E_2-E_1\leq N^{\fe}$, we take $E=(E_1+E_2)/2$ and the test function 
\begin{align*}
F_\varepsilon(x)=\bm 1_{[N(E_1-E), N(E_2-E)]}(x)*\phi(x/\varepsilon)/\varepsilon,
\end{align*}
Then Theorem \ref{t:bulkU} and \eqref{e:GbEvar2} together imply
\begin{align}\begin{split}\label{e:bound1}
&\phantom{{}={}}\bE\left[\left(|\{j:\la_j\in [E_1\sqrt{N}, E_2\sqrt{N}]\}|-N\int_{E_1}^{E_2}\rho_{sc}(x)\rd x\right)^2\right]\\
&=\lim_{\varepsilon\rightarrow 0}\bE\left[\left(\sum_i F_\varepsilon(\sqrt{N}\lambda_i-EN)-N\int F(x)\rho_{sc}(x)\rd x\right)^2\right]\\
&=\lim_{\varepsilon\rightarrow 0}\bE_{G\beta E}\left[\left(\sum_i F_\varepsilon(\sqrt{N}\lambda_i-EN)-N\int F(x)\rho_{sc}(x)\rd x\right)^2\right]+\oo(1)\\
&\lesssim (\log(2+N(E_2-E_1))^{\OO(1)}.
\end{split}\end{align}
If $E_2-E_1\geq N^{\fe-1}$, then it follows from Theorem \ref{t:rigid}, with high probability 
\begin{align}\label{e:bound2}
\left|\{j:\la_j\in [E_1\sqrt{N}, E_2\sqrt{N}]\}|-N\int_{E_1}^{E_2}\rho_{sc}(x)\rd x\right|\leq (\log N)^{\OO(1)}\lesssim \log (2+N(E_2-E_1)\wedge N)^{\OO(1)}.
\end{align}
The statement \eqref{e:var} follows from combining \eqref{e:bound1} and \eqref{e:bound2}.
\end{proof}

\subsection{Proof of Theorem \ref{t:bulk}}

Fix any $E\in (-2,2)$, we denote $j^{(s)}$ the smallest index such that $\lambda_{j^{(s)}}^{(s)}\geq E\sqrt{N}$. We rescale and relabel the eigenvalues 
\begin{align*}
\mu_j^{(s)}=\sqrt{N(4-E^2)}(\lambda^{(s)}_{j+j^{(s)}}-E\sqrt N), \quad 1\leq j+j^{(s)}\leq N+s,
\end{align*}
and relabel the weights
\begin{align*}
\gamma_j^{(s)}=|\langle \bmu_{j+j^{(s)}}, \bmh^{(s+1)}\rangle|^2, \quad  j+j^{(s)}\in \qq{1,N+s},\quad\gamma_j^{(s)}=1,\quad  j+j^{(s)}\not\in \qq{1,N+s},
\end{align*}
where $\bmh^{(s+1)}=(h_{1 N+s+1}, h_{2 N+s+1}, \cdots, h_{N+s N+s+1})^\top$ is independent of $H_{N+K}^{(s)}$.

With the new notations $\mu_j^{(s)}$ and $\gamma_j^{(s)}$, the point process is simply
\begin{align*}
\Theta_N^{(s)}=\sum_{j: 1\leq j+j^{(s)}\leq N+s}\delta_{\mu_j^{(s)}},
\end{align*}
where the labeling satisfies $\mu_{-1}^{(s)}<0\leq \mu_0^{(s)}$.  
We denote the doubly infinite sequence $\Gamma_N^{(s)}=\{\gamma_j^{(s)}\}_{j\in \bZ}$. It follows from the same argument as in Proposition \ref{p:eigenloc} that $\Gamma_N^{(s)}$ converges coordinatewise to a doubly infinite sequence with independent Gamma distributed entries, which has density $C_\beta x^{\beta/2-1}e^{-\beta x/2}$.
Given $\{\mu_j^{(s)}\}_{j\in \bZ}$ and $\{\gamma_j^{(s)}\}_{j\in \bZ}$, the point process $\Theta_N^{(s+1)}$ is the empirical zeros of the master equation \eqref{e:meq}
\begin{align}\label{e:recur}
\frac{h_{N+s+1N+s+1}}{\sqrt{N(4-E^2)}}-\frac{z}{N(4-E^2)}-\frac{E}{\sqrt{4-E^2}}-\sum_{j\in \bZ}\frac{\gamma_j^{(s)}}{\mu_j^{(s)}-z}=0.
\end{align}

Thanks to the bulk universality result, Theorem \ref{t:bulkU}, the point process $\Theta_N^{(s)}$ converges in distribution to the Sine$_\beta$ point process. Especially the family $\{\Theta_N^{(s)}\}_{0\leq s\leq K}$ is tight in the space of probability measures on $\cM(\bR)$. From the tightness, it is enough to prove that the law of the Markov process chain in \eqref{e:markov} is the only possible limit for a subsequence of the laws of $\{\Theta_N^{(s)}\}_{0\leq s\leq K}$. We define the following random variables $\{Y_N^{(s)}, Z_N^{(s)}\}_{0\leq s\leq K}$,
\begin{align*}
&Y_N^{(s)}=\sup_{x\in \bZ_{\geq 0}} (1+x)^{-3/4}\left(\Delta \Theta^{(s)}_{N}[0,x]+\Delta \Theta^{(s)}_{N}[-x,0]\right),\\
&Z_N^{(s)}=\sup_{x\in \bZ_{\geq 0}} (1+x^{\fd-1})^{-3/4}\left(
\left|\sum_{j\in\qq{x^{\fd},(x+1)^\fd}}(\gamma^{(s)}_j-1)\right|
+\left|
\sum_{-j\in\qq{x^\fd,(x+1)^{\fd}}}(\gamma^{(s)}_{j}-1)\right|
\right),
\end{align*}
where $\fd>0$ will be chosen later and
\begin{align*}
\Delta \Theta^{(s)}_{N}[a,b]=\left|\Theta_N^{(s)}[a,b]-N\int_{E+a/N\sqrt{4-E^2}}^{E+b/N\sqrt{4-E^2}}\rho_{sc}(x)\rd x\right|.
\end{align*}
The families $\{Y_N^{(s)}\}_{N\geq 1}$ $\{Z_N^{(s)}\}_{N\geq 1}$ are tight. Indeed, thanks to Corollary \ref{c:var}, we have
\begin{align*}
\bE\left[(Y_N^{(s)})^2\right]
&\leq \sum_{x\in \bZ_{\geq 0}}\frac{2}{(1+x)^{3/2}}\left(\bE\left(\Delta \Theta^{(s)}_{N}[0,x]\right)^2+\bE\left(\Delta \Theta^{(s)}_{N}[-x,0]\right)^2\right)\\
&\lesssim\sum_{x\in \bZ_{\geq 0}, x\leq N^\fe}\frac{\log (2+x)^{\OO(1)}}{(1+x)^{3/2}}+\sum_{x\in \bZ_{\geq 0}, x\geq N^\fe}\frac{\log (N)^{\OO(1)}}{(1+x)^{3/2}}\lesssim 1.
\end{align*}
For $Z_N^{(s)}$, similarly, we have
\begin{align}\label{e:ZNb}
\bE\left[(Z_N^{(s)})^2\right]
&\leq \sum_{x\in \bZ_{\geq0}}\frac{2}{(1+x^{\fd-1})^{3/2}}
\left(
\bE\left|\sum_{j\in\qq{x^{\fd},(x+1)^\fd}}(\gamma^{(s)}_j-1)\right|^2
+\bE\left|
\sum_{-j\in\qq{x^\fd,(x+1)^{\fd}}}(\gamma^{(s)}_{j}-1)\right|^2.
\right)
\end{align}
To compute the expectations on the righthand side of \eqref{e:ZNb}, we can first integrate out $\bmh^{(s+1)}$,
\begin{align*}
\bE[(\gamma_j^{(s)}-1)(\gamma_{j'}^{(s)}-1)]
=2/\beta\delta_{jj'}+\OO\left(\bE\left[\sum_{\al}|\bmu^{(s)}_{j\al}|^2|\bmu^{(s)}_{j'\al}|^2\right]\right)=2/\beta\delta_{jj'}+\OO\left(\frac{(\log N)^{\OO(1)}}{N}\right),
\end{align*}
where in the last equality, we used \eqref{e:delocalize} that the eigenvectors are delocalized with high probability. Therefore, it follows that for any $x\lesssim N^{1/\fd}$,
\begin{align}\begin{split}\label{e:bd2}
\bE\left|\sum_{j\in\qq{x^{\fd},(x+1)^\fd}}(\gamma^{(s)}_j-1)\right|^2
+\bE\left|
\sum_{-j\in\qq{x^\fd,(x+1)^{\fd}}}(\gamma^{(s)}_{j}-1)\right|^2
&\lesssim \sum_{j,j'\in\qq{x^{\fd},(x+1)^\fd}}\beta\delta_{jj'}+\frac{(\log N)^{\OO(1)}}{N}\\
&\lesssim x^{\fd-1}+\frac{(\log N)^{\OO(1)}}{N}x^{2(\fd-1)}\lesssim x^{\fd-1}.
\end{split}\end{align}
We plug \eqref{e:bd2} into \eqref{e:ZNb}, 
\begin{align*}
\bE\left[(Z_N^{(s)})^2\right]
&\lesssim \sum_{x\in \bZ_{\geq 0}}\frac{x^{\fd-1}}{(1+x^{\fd-1})^{3/2}}\lesssim 1.
\end{align*}

As a consequence, $\{\Theta^{(s)}_N,\Gamma_N^{(s)}, Z_N^{(s)}, Y_N^{(s)}\}_{0\leq s\leq K}$ form a tight family of random variables. We can find a subsequence for which they converge in law. By Skorokhod representation theorem, this family of random variables $\{\Theta^{(s)}_N, \Gamma_N^{(s)}, Z_N^{(s)}, Y_N^{(s)}\}_{0\leq s\leq K}$ has the same law as some family $\{\tilde \Theta^{(s)}_N, \tilde\Gamma_N^{(s)}, \tilde Z_N^{(s)}, \tilde Y_N^{(s)}\}_{0\leq s\leq K}$, which converges almost surely to $\{\Theta^{(s)}, \Gamma^{(s)}, Z^{(s)}, Y^{(s)}\}_{0\leq s\leq K}$. 
In the following we prove that the law of the Markov process chain in \eqref{e:markov} is the only possible law for 
$\{\Theta^{(s)}\}_{0\leq s\leq K}$.

For simplicity of notations, we denote
\begin{align}\label{e:notation}
\tilde \Theta^{(s)}_N=\sum_{j}\delta_{\mu_j^{N}},\quad
\tilde \Gamma_N^{(s)}=\{\gamma_j^{N}\}_{j\in \bZ},\quad
\Theta^{(s)}=\sum_{j}\delta_{\mu_j},\quad
\Gamma^{(s)}=\{\gamma_j\}_{j\in \bZ}.
\end{align}
By the same argument as in Proposition \ref{p:eigenloc}, $\Gamma^{(s)}=\{\gamma_j\}_{j\in \bZ}$ are independent, and $\gamma_j$ follows Gamma distribution with density $C_\beta x^{\beta/2-1}e^{-\beta x/2}$.
Conditionally on $\Theta_N^{(s)}$, the new point process $\Theta_N^{(s+1)}$ is the empirical zeros of the equation \eqref{e:recur}. By our construction of $\{\tilde \Theta^{(s)}_N, \tilde\Gamma_N^{(s)}, \tilde Z_N^{(s)}, \tilde Y_N^{(s)}\}_{0\leq s\leq K}$ which has the same law as $\{\Theta^{(s)}_N, \Gamma_N^{(s)}, Z_N^{(s)}, Y_N^{(s)}\}_{0\leq s\leq K}$. Therefore, with the notations introduced in \eqref{e:notation}, the point process $\tilde \Theta_N^{(s+1)}$ is the empirical zeros of the equation
\begin{align}\label{e:recur2}
\frac{h_{N+s+1N+s+1}}{\sqrt{N(4-E^2)}}-\frac{z}{N(4-E^2)}-\frac{E}{\sqrt{4-E^2}}-\sum_{j\in \bZ}\frac{\gamma_j^{N}}{\mu_j^{N}-z}=0.
\end{align}
Since by our construction $\{\tilde \Theta^{(s)}_N, \tilde\Gamma_N^{(s)}, \tilde Z_N^{(s)}, \tilde Y_N^{(s)}\}_{0\leq s\leq K}$ converges almost surely to $\{\Theta^{(s)}, \Gamma^{(s)}, Z^{(s)}, Y^{(s)}\}_{0\leq s\leq K}$, the following holds almost surely,
\begin{align}\label{e:alcov}
\lim_{N\rightarrow \infty}\mu_j^{N}=\mu_j, \quad
\lim_{N\rightarrow \infty}\gamma_j^{N}=\gamma_j.
\end{align}
In the following we prove 

\begin{proposition}\label{p:limit}
Under the assumptions of Theorem \ref{t:bulk}, if $\{\tilde \Theta^{(s)}_N, \tilde\Gamma_N^{(s)}, \tilde Z_N^{(s)}, \tilde Y_N^{(s)}\}_{0\leq s\leq K}$ converges almost surely to $\{\Theta^{(s)}, \Gamma^{(s)}, Z^{(s)}, Y^{(s)}\}_{0\leq s\leq K}$, using the notations introduced in \eqref{e:notation}, then we have
 almost surely
\begin{align*}
\lim_{N\rightarrow\infty}\frac{h_{N+s+1N+s+1}}{\sqrt{N(4-E^2)}}-\frac{z}{N(4-E^2)}-\frac{E}{\sqrt{4-E^2}}-\sum_{j\in \bZ}\frac{\gamma_j^{N}}{\mu_j^{N}-z}
=-\sum_{j\in \bZ}\frac{\gamma_j}{\mu_j-z}-\frac{E}{2\sqrt{4-E^2}} ,
\end{align*}
on $z\in \bR\setminus\{\mu_j\}_{j\in \bZ}$.
\end{proposition}
Then it follows from \eqref{e:recur2} and Proposition \ref{p:limit} that the new point process $\Theta^{(s+1)}$ is the empirical zeros of equation
\begin{align*}
-\frac{E}{2\sqrt{4-E^2}}-\sum_{j\in \bZ}\frac{\gamma_j}{\mu_j-z}=0.
\end{align*}
Especially, we have
\begin{align*}
\Theta^{(s+1)}=S^{-1}_{\sum_{j \in \bR }\gamma_j\delta_{\mu_j}}\left(-\frac{E}{2\sqrt{4-E^2}}\right)=\cD(\Theta^{(s)}, \Gamma^{(s)},h),\quad h=-\frac{E}{2\sqrt{4-E^2}},
\end{align*}
and $\Gamma^{(s)}=\{\gamma_j\}_{j\in \bZ}$ are independent, with Gamma distribution with density $C_\beta x^{\beta/2-1}e^{-\beta x/2}$.
We can iterate the argument above for $s=0,1,2,\cdots, K-1$, to conclude that $\{\Theta^{(s)}_N\}_{0\leq s\leq K}$ forms a Markov chain and converges in distribution to the Markov chain \eqref{e:markov} for the level $h=-E/2\sqrt{4-E^2}$.

%

\begin{proof}[Proof of Proposition \ref{p:limit}]
Since $h_{N+s+1N+s+1}$ has bounded fourth moment, almost surely, 
\begin{align*}
\lim_{N\rightarrow\infty}\frac{h_{N+s+1N+s+1}}{\sqrt{N(4-E^2)}}-\frac{z}{N(4-E^2)}=0.
\end{align*}
It remains to show that 
\begin{align}\label{e:claim}
\lim_{N\rightarrow\infty}\sum_{j\in \bZ}\frac{\gamma_j^{N}}{\mu_j^{N}-z}
=\sum_{j\in \bZ}\frac{\gamma_j}{\mu_j-z}-\frac{E}{2\sqrt{4-E^2}}.
\end{align}

Almost surely $\tilde Z_N^{(s)}, \tilde Y_N^{(s)}, Z^{(s)}, Y^{(s)}$ are all bounded. Especially for any $x\in \bZ_{\geq 0}$, it holds
\begin{align}\begin{split}\label{e:prelimbound}
&\left||\{j:\mu_j^N\in [0,x]\}|-N\mu_{sc}\left[ E, E+\frac{x}{N\sqrt{4-E^2}}\right]\right|\lesssim (1+x)^{3/4},\\
&\left||\{j:\mu_j^N\in [-x,0]\}|-N\mu_{sc}\left[E-\frac{x}{N\sqrt{4-E^2}},E\right]\right|\lesssim (1+x)^{3/4},\\
&\left|\sum_{j\in\qq{x^{\fd},(x+1)^\fd}}(\gamma^N_j-1)\right|
,\left|
\sum_{-j\in\qq{x^\fd,(x+1)^{\fd}}}(\gamma^N_{j}-1)\right|
\lesssim (1+x^{\fd-1})^{3/4},
\end{split}\end{align} 
and similar estimates hold when we replace $\{\mu_j^N\}_{j\in \bZ}, \{\gamma_j^N\}_{j\in \bZ}$ by $\{\mu_j\}_{j\in \bZ}, \{\gamma_j\}_{j\in \bZ}$,
\begin{align}\begin{split}\label{e:limbound}
&\left||\{j:\mu_j\in [0,x]\}|-x/2\pi\right|,\left||\{j:\mu_j\in [-x,0]\}|-x/2\pi\right|\lesssim (1+x)^{3/4},\\
&\left|\sum_{j\in\qq{x^{\fd},(x+1)^\fd}}(\gamma_j-1)\right|
,\left|
\sum_{-j\in\qq{x^\fd,(x+1)^{\fd}}}(\gamma_{j}-1)\right|
\lesssim (1+x^{\fd-1})^{3/4}.
\end{split}\end{align}

Fix some large integer $\fj>0$, we denote
\begin{align*}
S_{N,\frak j}(z)=\sum_{|j|\leq \fj}\frac{\gamma_{j}^{N}}{\mu_{j}^{N}-z},\quad
S_{\fj}(z)=\sum_{|j|\leq \fj}\frac{\gamma_j}{\mu_j-z}.
\end{align*}
Then thanks to \eqref{e:alcov} we have
\begin{align}\label{e:divj}
\lim_{N\rightarrow \infty}S_{N,\fj}(z)=S_{\fj}(z).
\end{align}
From \eqref{e:limbound}, we have $\{\mu_j\}_{j\in \bZ}\in \cL$ and $\{\gamma_j\}_{j\in \bZ}\in\cG$ ($\cL$ and $\cG$ are defined in Section \ref{s:main}). It follows from \cite[Theorem 3]{NaVi} the Stieltjes transform of the weighted measure $\sum_{j\in \bZ}\gamma_j \delta_{\mu_j}$ is well-defined,
\begin{align}\label{e:jlim}
\lim_{\fj\rightarrow \infty}S_\fj(z)=\sum_{j\in \bZ} \frac{\gamma_j}{\mu_j-z}.
\end{align}
In the rest, we show that
\begin{align}\label{e:large}
\left|\sum_{|j|\geq \fj}\frac{\gamma_{j}^{N}}{\mu_{j}^{N}-z}-\frac{1}{\sqrt{4-E^2}}\int_{|x-E|\geq 2\pi \fj/N\sqrt{4-E^2}}\frac{\rho_{sc}(x)}{x-z}\rd x \right|\lesssim \fj^{-(1-1/\fd)/4}+\fj^{-1/\fd}.
\end{align}
Then the statement \eqref{e:claim} follows from combining \eqref{e:divj}, \eqref{e:jlim} and \eqref{e:large},
\begin{align*}\begin{split}
&\lim_{N\rightarrow \infty}\sum_{j\in \bZ}\frac{\gamma_j^{N}}{\mu_j^{N}-z}
=\lim_{\fj\rightarrow\infty}\lim_{N\rightarrow \infty}S_{N,\fj}(z)
+\sum_{|j|\geq \fj}\frac{\gamma_{j}^{N}}{\mu_{j}^{N}-z}\\
&=\lim_{\fj\rightarrow\infty} S_\fj(z)+\frac{1}{\sqrt{4-E^2}}
\int_{|x-E|\geq 2\pi \fj/N\sqrt{4-E^2}}\frac{\rho_{sc}(x)}{x-z}\rd x\\
&=\sum_{j\in \bZ}\frac{\gamma_j}{\mu_j-z}+
P.V.\frac{1}{\sqrt{4-E^2}}\int\frac{\rho_{sc}(x)}{x-z}\rd x
=\sum_{j\in \bZ}\frac{\gamma_j}{\mu_j-z}-\frac{E}{2\sqrt{4-E^2}}.
\end{split}\end{align*}

In the following we prove \eqref{e:large}. Thanks to \eqref{e:prelimbound}, it holds
\begin{align*}
|\{j:\mu^N_j \in [0,x]\}|=N\int_0^{x/N\sqrt{4-E^2}}\rho_{sc}(E+x)\rd x+\OO((1+x)^{3/4})\asymp 1+x.
\end{align*}
As a consequence, for any fixed $x\in \bZ_{\geq 0}$, as $N$ goes to infinity, 
\begin{align}\label{e:countbound}
|\{j:\mu^N_j \in [0,x]\}|=\frac{x}{2\pi}+\OO((1+x)^{3/4})+\oo(1).
\end{align}
By plugging $x=2\pi \fj\pm C\fj^{3/4}$ for some $C>0$ and $\fj>0$ sufficiently large in \eqref{e:countbound}, we get the estimates
\begin{align}\label{e:mjbound}
\mu_\fj^N=2\pi \fj+\OO(\fj^{3/4}),
\end{align}
and by symmetry, the statement \eqref{e:mjbound} holds also for $\fj<0$. We split the lefthand side of \eqref{e:large} in two parts,
\begin{align}\begin{split}\label{e:two}
&\phantom{{}={}}\sum_{|j|\geq \fj}\frac{\gamma_{j}^{N}}{\mu_{j}^{N}-z}-\frac{1}{\sqrt{4-E^2}}\int_{|x-E|\geq 2\pi \fj/N\sqrt{4-E^2}}\frac{\rho_{sc}(x)}{x-z}\rd x\\
&=\sum_{|j|\geq \fj}\frac{\gamma_{j}^{N}-1}{\mu_{j}^{N}-z}
+\sum_{|j|\geq \fj}\frac{1}{\mu_{j}^{N}-z}
-\frac{1}{\sqrt{4-E^2}}\int_{|x-E|\geq 2\pi \fj/N\sqrt{4-E^2}}\frac{\rho_{sc}(x)}{x-z}\rd x.
\end{split}\end{align}
For the first term in \eqref{e:two}, we estimate the sum corresponding to $j\geq \fj$. The sum corresponding to $j\leq -\fj$ can be estimated in exactly the same way.
\begin{align}\begin{split}\label{e:sumerror}
&\phantom{{}={}}\sum_{j\geq \fj}\frac{\gamma_{j}^{N}-1}{\mu_{j}^{N}-z}
=\sum_{x\geq \fj^{1/\fd}}\sum_{j\in \qq{x^\fd, (x+1)^\fd}}\frac{\gamma_{j}^{N}-1}{\mu_{j}^{N}-z}\\
&=\sum_{x\geq \fj^{1/\fd}}\sum_{j\in \qq{x^\fd, (x+1)^\fd}}\frac{\gamma_{j}^{N}-1}{\mu_{\floor{x^\fd}}^{N}-z}+\frac{(\gamma_{j}^{N}-1)(\mu_{\floor{x^\fd}}^N-\mu_j^N)}{(\mu_j^N-z)(\mu_{\floor{x^\fd}}^{N}-z)}
\end{split}\end{align}
For the first term on the righthand side of \eqref{e:sumerror}, we estimate it using \eqref{e:prelimbound}
 and \eqref{e:mjbound} 
\begin{align}\label{e:tm1}
\left|\sum_{x\geq \fj^{1/\fd}}\sum_{j\in \qq{x^\fd, (x+1)^\fd}}\frac{\gamma_{j}^{N}-1}{\mu_{\floor{x^\fd}}^{N}-z}\right|
\lesssim
\sum_{x\geq \fj^{1/\fd}}\frac{(1+x^{\fd-1})^{3/4}}{x^{\fd}}\lesssim \fj^{-(1-1/\fd)/4}
\end{align}
For the second term on the righthand side of \eqref{e:sumerror}, we estimate it using \eqref{e:prelimbound} and \eqref{e:mjbound}  again
\begin{align}\begin{split}\label{e:tm2}
&\phantom{{}={}}\left|\sum_{x\geq \fj^{1/\fd}}\sum_{j\in \qq{x^\fd, (x+1)^\fd}}\frac{(\gamma_{j}^{N}-1)(\mu_{\floor{x^\fd}}^N-\mu_j^N)}{(\mu_j^N-z)(\mu_{\floor{x^\fd}}^{N}-z)}\right|
\lesssim 
\sum_{x\geq \fj^{1/\fd}}\sum_{j\in \qq{x^\fd, (x+1)^\fd}}\frac{(\gamma_{j}^{N}+1)(\mu_{\floor{(x+1)^\fd}}^N-\mu_{\floor{x^\fd}}^N)}{x^{2\fd}}\\
&\lesssim
\sum_{x\geq \fj^{1/\fd}}\frac{(\mu_{\floor{(x+1)^\fd}}^N-\mu_{\floor{x^\fd}}^N)}{x^{1+\fd}}
\leq \sum_{x\geq \fj^{1/\fd}}\mu_{\floor{x^\fd}}^N\left(\frac{1}{x^{1+\fd}}-\frac{1}{(x+1)^{1+\fd}}\right)
\leq \sum_{x\geq \fj^{1/\fd}}\frac{1}{x^{2}}\lesssim \fj^{-1/\fd}.
\end{split}\end{align}
We obtain an upper bound for \eqref{e:sumerror} by combining \eqref{e:tm1} and \eqref{e:tm2},
\begin{align}\label{e:mid}
\left|\sum_{j\geq \fj}\frac{\gamma_{j}^{N}-1}{\mu_{j}^{N}-z}
\right|\lesssim \fj^{-(1-1/\fd)/4}+\fj^{-1/\fd}.
\end{align}

For the second term in \eqref{e:two}, we first estimate the sum corresponding to $j\geq \fj$. Performing an integration by part
\begin{align*}\begin{split}
&\phantom{{}={}}\sum_{j\geq \fj}\frac{1}{\mu_{j}^{N}-z}
-\frac{1}{\sqrt{4-E^2}}\int_{E+\mu_\fj^N/N\sqrt{4-E^2}}^\infty \frac{\rho_{sc}(x)}{x-z}\rd x\\
&=\int_{\mu_{\fj}^N}^\infty\frac{N\mu_{sc}\left[E+\frac{\mu_\fj^N}{N\sqrt{4-E^2}} , E+\frac{x}{N\sqrt{4-E^2}}\right]-|\{j: \mu_j^N \in [\mu_\fj^N, x]\}|}{(x-z)^2}\rd x
\end{split}\end{align*}
Thanks to the estimates \eqref{e:prelimbound} and \eqref{e:mjbound} 
\begin{align*}
\mu_\fj^N=2\pi \fj+\OO(\fj^{3/4})=\OO(\fj), \quad \left|N\int_{\mu_\fj^N/N\sqrt{4-E^2} }^{x/N\sqrt{4-E^2}}\rho_{sc}(E+x)\rd x-|\{j: \mu_j^N \in [\mu_\fj^N, x]\}|\right|\lesssim (1+x)^{3/4}.
\end{align*}
For $\fj\gg |z|$, we have the following bound
\begin{align}\label{e:tm3}
\left|\sum_{j\geq \fj}\frac{1}{\mu_{j}^{N}-z}
-\frac{1}{\sqrt{4-E^2}}\int_{E+\mu_\fj^N/N\sqrt{4-E^2}}^\infty \frac{\rho_{sc}(x)}{x-z}\rd x\right|
\lesssim\int_{\mu_{\fj}^N}^\infty\frac{(1+x)^{3/4}}{(x-z)^2}\rd x\lesssim \fj^{-1/4}.
\end{align}
By exactly the same argument, we also have 
\begin{align}\label{e:tm4}
\left|\sum_{j\leq -\fj}\frac{1}{\mu_{j}^{N}-z}
-\frac{1}{\sqrt{4-E^2}}\int_{-\infty}^{E+\mu_{-\fj}^N/N\sqrt{4-E^2}}\frac{\rho_{sc}(x)}{x-z}\rd x\right|
\lesssim\int_{-\infty}^{\mu_{-\fj}^N}\frac{(1+x)^{3/4}}{(x-z)^2}\rd x\lesssim \fj^{-1/4}.
\end{align}
The claim \eqref{e:large} follows from combining \eqref{e:two}, \eqref{e:mid}, \eqref{e:tm3} and \eqref{e:tm4}. This finishes the proof of Theorem \ref{t:bulk}.

\end{proof}

\bibliography{References.bib}{}

\begin{thebibliography}{10}

\bibitem{Open}
Open problems: {AIM} workshop on random matrices(dec 2010).
\newblock \url{https://aimath.org/WWN/randommatrices/randommatrices.pdf}.
\newblock Accessed: 2019-7-17.

\bibitem{MR3949269}
A.~Adhikari and Z.~Che.
\newblock Edge universality of correlated {G}aussians.
\newblock {\em Electron. J. Probab.}, 24:Paper No. 44, 25, 2019.

\bibitem{MR3149438}
M.~Adler, P.~van Moerbeke, and D.~Wang.
\newblock Random matrix minor processes related to percolation theory.
\newblock {\em Random Matrices Theory Appl.}, 2(4):1350008, 72, 2013.

\bibitem{MR3916110}
A.~Aggarwal.
\newblock Bulk universality for generalized {W}igner matrices with few moments.
\newblock {\em Probab. Theory Related Fields}, 173(1-2):375--432, 2019.

\bibitem{ALY}
A.~Aggarwal and H.-T.~Y. Patrick~Lopatto.
\newblock {GOE} statistics for {L}evy matrices.
\newblock {\em preprint: arXiv: 1806.07363}, 2018.

\bibitem{AEKS}
J.~Alt, L.~Erd{\H o}s, T.~Kr{\"u}ger, and D.~Schr{\"o}der.
\newblock Correlated random matrices: Band rigidity and edge universality.
\newblock {\em preprint: arXiv: 1804.07744}, 2018.

\bibitem{MR1818248}
Y.~Baryshnikov.
\newblock G{UE}s and queues.
\newblock {\em Probab. Theory Related Fields}, 119(2):256--274, 2001.

\bibitem{MR3729611}
R.~Bauerschmidt, J.~Huang, A.~Knowles, and H.-T. Yau.
\newblock Bulk eigenvalue statistics for random regular graphs.
\newblock {\em Ann. Probab.}, 45(6A):3626--3663, 2017.

\bibitem{MR3541852}
P.~Bourgade, L.~Erd\H{o}s, H.-T. Yau, and J.~Yin.
\newblock Fixed energy universality for generalized {W}igner matrices.
\newblock {\em Comm. Pure Appl. Math.}, 69(10):1815--1881, 2016.

\bibitem{MR2489161}
C.~Boutillier.
\newblock The bead model and limit behaviors of dimer models.
\newblock {\em Ann. Probab.}, 37(1):107--142, 2009.

\bibitem{MR3629874}
Z.~Che.
\newblock Universality of random matrices with correlated entries.
\newblock {\em Electron. J. Probab.}, 22:Paper No. 30, 38, 2017.

\bibitem{MR3941370}
L.~Erd\H{o}s, T.~Kr\"{u}ger, and D.~Schr\"{o}der.
\newblock Random matrices with slow correlation decay.
\newblock {\em Forum Math. Sigma}, 7:e8, 89, 2019.

\bibitem{MR2964770}
L.~Erd{\H{o}}s, A.~Knowles, H.-T. Yau, and J.~Yin.
\newblock Spectral statistics of {E}rd{\H o}s-{R}\'enyi {G}raphs {II}:
  {E}igenvalue spacing and the extreme eigenvalues.
\newblock {\em Comm. Math. Phys.}, 314(3):587--640, 2012.

\bibitem{MR2662426}
L.~Erd{\H{o}}s, S.~P{\'e}ch{\'e}, J.~A. Ramirez, B.~Schlein, and H.-T. Yau.
\newblock Bulk universality for {W}igner matrices.
\newblock {\em Comm. Pure Appl. Math.}, 63(7):895--925, 2010.

\bibitem{MR2661171}
L.~Erd{\H{o}}s, J.~Ramirez, B.~Schlein, T.~Tao, V.~Vu, and H.-T. Yau.
\newblock Bulk universality for {W}igner {H}ermitian matrices with
  subexponential decay.
\newblock {\em Math. Res. Lett.}, 17(4):667--674, 2010.

\bibitem{MR2481753}
L.~Erd{\H{o}}s, B.~Schlein, and H.-T. Yau.
\newblock Local semicircle law and complete delocalization for {W}igner random
  matrices.
\newblock {\em Comm. Math. Phys.}, 287(2):641--655, 2009.

\bibitem{MR2810797}
L.~Erd{\H{o}}s, B.~Schlein, and H.-T. Yau.
\newblock Universality of random matrices and local relaxation flow.
\newblock {\em Invent. Math.}, 185(1):75--119, 2011.

\bibitem{MR2919197}
L.~Erd{\H{o}}s, B.~Schlein, H.-T. Yau, and J.~Yin.
\newblock The local relaxation flow approach to universality of the local
  statistics for random matrices.
\newblock {\em Ann. Inst. Henri Poincar\'e Probab. Stat.}, 48(1):1--46, 2012.

\bibitem{MR3372074}
L.~Erd{\H{o}}s and H.-T. Yau.
\newblock Gap universality of generalized {W}igner and {$\beta$}-ensembles.
\newblock {\em J. Eur. Math. Soc. (JEMS)}, 17(8):1927--2036, 2015.

\bibitem{MR2847916}
L.~Erd{\H{o}}s, H.-T. Yau, and J.~Yin.
\newblock Universality for generalized {W}igner matrices with {B}ernoulli
  distribution.
\newblock {\em J. Comb.}, 2(1):15--81, 2011.

\bibitem{MR2981427}
L.~Erd{\H{o}}s, H.-T. Yau, and J.~Yin.
\newblock Bulk universality for generalized {W}igner matrices.
\newblock {\em Probab. Theory Related Fields}, 154(1-2):341--407, 2012.

\bibitem{MR2871147}
L.~Erd{\H{o}}s, H.-T. Yau, and J.~Yin.
\newblock Rigidity of eigenvalues of generalized {W}igner matrices.
\newblock {\em Adv. Math.}, 229(3):1435--1515, 2012.

\bibitem{MR3418747}
V.~Gorin and M.~Shkolnikov.
\newblock Multilevel {D}yson {B}rownian motions via {J}ack polynomials.
\newblock {\em Probab. Theory Related Fields}, 163(3-4):413--463, 2015.

\bibitem{MR3558206}
V.~Gorin and M.~Shkolnikov.
\newblock Interacting particle systems at the edge of multilevel {D}yson
  {B}rownian motions.
\newblock {\em Adv. Math.}, 304:90--130, 2017.

\bibitem{GNT}
F.~G{\"o}tze, A.~Naumov, and A.~Tikhomirov.
\newblock Local semicircle law under fourth moment condition.
\newblock {\em to appear in J. Theor. Probab.}, 2019.

\bibitem{MR3429490}
J.~Huang, B.~Landon, and H.-T. Yau.
\newblock Bulk universality of sparse random matrices.
\newblock {\em J. Math. Phys.}, 56(12):123301, 19, 2015.

\bibitem{MR1810949}
K.~Johansson.
\newblock Universality of the local spacing distribution in certain ensembles
  of {H}ermitian {W}igner matrices.
\newblock {\em Comm. Math. Phys.}, 215(3):683--705, 2001.

\bibitem{MR2268547}
K.~Johansson and E.~Nordenstam.
\newblock Eigenvalues of {GUE} minors.
\newblock {\em Electron. J. Probab.}, 11:no. 50, 1342--1371, 2006.

\bibitem{MR3914908}
B.~Landon, P.~Sosoe, and H.-T. Yau.
\newblock Fixed energy universality of {D}yson {B}rownian motion.
\newblock {\em Adv. Math.}, 346:1137--1332, 2019.

\bibitem{MR3405746}
J.~O. Lee and K.~Schnelli.
\newblock Edge universality for deformed {W}igner matrices.
\newblock {\em Rev. Math. Phys.}, 27(8):1550018, 94, 2015.

\bibitem{MR3502606}
J.~O. Lee, K.~Schnelli, B.~Stetler, and H.-T. Yau.
\newblock Bulk universality for deformed {W}igner matrices.
\newblock {\em Ann. Probab.}, 44(3):2349--2425, 2016.

\bibitem{MR3161313}
J.~O. Lee and J.~Yin.
\newblock A necessary and sufficient condition for edge universality of
  {W}igner matrices.
\newblock {\em Duke Math. J.}, 163(1):117--173, 2014.

\bibitem{MR2129906}
M.~L. Mehta.
\newblock {\em Random matrices}, volume 142 of {\em Pure and Applied
  Mathematics (Amsterdam)}.
\newblock Elsevier/Academic Press, Amsterdam, third edition, 2004.

\bibitem{NaVi}
J.~Najnudel and B.~Vir{\'a}g.
\newblock The bead process for beta ensemble.
\newblock {\em preprint: arXiv: 1904.00848}, 2019.

\bibitem{NaVi2}
J.~Najnudel and B.~Vir{\'a}g.
\newblock Uniform point variance bounds in classical beta ensembles.
\newblock {\em preprint: arXiv: 1904.00858}, 2019.

\bibitem{MR1647832}
Y.~G. Sinai and A.~B. Soshnikov.
\newblock A refinement of {W}igner's semicircle law in a neighborhood of the
  spectrum edge for random symmetric matrices.
\newblock {\em Funktsional. Anal. i Prilozhen.}, 32(2):56--79, 96, 1998.

\bibitem{MR3729037}
S.~Sodin.
\newblock Several applications of the moment method in random matrix theory.
\newblock In {\em Proceedings of the {I}nternational {C}ongress of
  {M}athematicians---{S}eoul 2014. {V}ol. {III}}, pages 451--475. Kyung Moon
  Sa, Seoul, 2014.

\bibitem{MR3403994}
S.~Sodin.
\newblock A limit theorem at the spectral edge for corners of time-dependent
  {W}igner matrices.
\newblock {\em Int. Math. Res. Not. IMRN}, (17):7575--7607, 2015.

\bibitem{MR1727234}
A.~Soshnikov.
\newblock Universality at the edge of the spectrum in {W}igner random matrices.
\newblock {\em Comm. Math. Phys.}, 207(3):697--733, 1999.

\bibitem{MR2669449}
T.~Tao and V.~Vu.
\newblock Random matrices: universality of local eigenvalue statistics up to
  the edge.
\newblock {\em Comm. Math. Phys.}, 298(2):549--572, 2010.

\bibitem{MR3306005}
T.~Tao and V.~Vu.
\newblock Random matrices: universality of local spectral statistics of
  non-{H}ermitian matrices.
\newblock {\em Ann. Probab.}, 43(2):782--874, 2015.

\bibitem{MR1257246}
C.~A. Tracy and H.~Widom.
\newblock Level-spacing distributions and the {A}iry kernel.
\newblock {\em Comm. Math. Phys.}, 159(1):151--174, 1994.

\bibitem{MR1385083}
C.~A. Tracy and H.~Widom.
\newblock On orthogonal and symplectic matrix ensembles.
\newblock {\em Comm. Math. Phys.}, 177(3):727--754, 1996.

\bibitem{MR0077805}
E.~P. Wigner.
\newblock Characteristic vectors of bordered matrices with infinite dimensions.
\newblock {\em Ann. of Math. (2)}, 62:548--564, 1955.

\bibitem{MR0083848}
E.~P. Wigner.
\newblock Characteristic vectors of bordered matrices with infinite dimensions.
  {II}.
\newblock {\em Ann. of Math. (2)}, 65:203--207, 1957.

\end{thebibliography}
\bibliographystyle{abbrv}

\end{document}